\documentclass[amsthm]{elsart}

\makeatletter
\def\@journal{}
\let\@date\@empty

\makeatother

\usepackage{yjsco}
\usepackage{natbib}
\usepackage{fleqn}
\usepackage{amssymb,latexsym}
\usepackage{footnote}

\usepackage{ulem}
\usepackage{xcolor}

\newtheorem{rmk}{Remark}
\newtheorem{exa}{Example}
\newtheorem{question}{Question}
\newtheorem{conjecture}{Conjecture}

\begin{document}
\begin{frontmatter}

\title{On the cactus rank of cubics forms}

\thanks{The first author was partially supported by  Project Galaad of INRIA So\-phia Antipolis M\'editerran\'ee  (France)  and  by Marie Curie Intra-European Fellowships for Career Development (FP7-PEOPLE-2009-IEF): ``DECONSTRUCT".}

\author{Alessandra Bernardi}
\address{
Dipartimento di Matematica, Universit\`a di Trento, via Sommarive 14, 38123, Trento, Italy.}
\ead{alessandra.bernardi@unitn.it}
\ead[url]{https://sites.google.com/unitn.it/alessandra-bernardi/home}

\author{Kristian Ranestad}
\address{Matematisk institutt, Universitetet i Oslo, PO Box 1053, Blindern, NO-0316 Oslo, Norway.}
\ead{ranestad@math.uio.no}
\ead[url]{http://folk.uio.no/ranestad/}

\begin{abstract}
We prove that the smallest degree of an apolar $0$-dimensional scheme of a general cubic form in $n+1$ variables is at most $2n+2$, when $n\geq 8$,  and therefore smaller than the rank of the form. For the general reducible cubic form  the smallest degree of an apolar subscheme is $n+2$, while the rank is at least $2n$.
\end{abstract}


\end{frontmatter}

\section*{Introduction}
The \emph{rank} of a homogeneous form $F\in S:= {\mathbb C}[x_0,\ldots ,x_n]$ of degree $d$ is the minimal number of linear forms $L_1,\ldots,L_r$ needed to write $F$ as a sum of pure $d$-powers:
\[
F=L_1^d+\cdots+L_r^d.
\]
Various other notions of rank, such as \emph{border rank} and \emph{cactus rank}, appear in the study of higher secant varieties and are closely related to the rank.
 The cactus rank is the minimal length of an apolar subscheme to $F$, while the border rank is the minimal $r$ such that $F$ is a limit of forms of rank $r$. For an extensive description and usage of the classical concept of \emph{apolarity}, we refer to \citep{IK} and \citep{RS} and the references therein, which go back to the late XIX century with A. Clebsch, J. L\"uroth, T. Reye, G. Scorza and to the beginning of the XX century with E. Lasker, F. H. S. Macaulay, J. J. Sylvester, A. Terracini and E. K. Wakeford. 

The notion of cactus rank is recent and  coincides with scheme length introduced by Iarrobino and Kanev in \citep{IK}.  We use the name cactus rank to make the association to \emph{cactus varieties} introduced in \citep{BB} in a study of higher secant varieties.  

The cactus rank and the border rank are both less than or equal to the rank as is explained in Section \ref{AGs}, while a natural lower bound for both of them is the differential length (also called the Hankel rank),  the maximum of the dimensions of the 
space of $k$-th order partials of $F$ as $k$ varies between $0$ and 
$d$.   For a general form the rank and the border rank coincide, but little is known about the cactus rank beyond these bounds, cf. \citep{IK}.

 For specific forms, more is known: For irreducible forms that do not define a cone, the cactus 
 rank is minimal for forms of Fermat type, e.g. $F=x_0^d+\cdots+x_n^d$. 
 In this case the rank coincides with the Hankel rank and hence also with the cactus rank and the border rank.

The first main result of this paper, Theorem \ref{ND}, is that for large $n$ and $d$, the cactus rank of a general form is strictly less than the rank. 

For cubic forms we give more specific results: In Section \ref{cubics} we show that there are cubic forms with 
 minimal cactus rank (equal to the Hankel rank) whose border rank is strictly higher and compute the cactus rank of a general reducible cubic form. 
 
 The rank of forms has seen growing interest in recent years.  Any apolar subscheme to $F$ of minimal length is locally Gorenstein (\citep[proof of Lemma 2.4]{BB}), therefore this work is close in line to \citep{I}, \citep{IK} and \citep {ER}, in their study of apolarity and the  local Gorenstein algebra associated to a polynomial.   Applications to higher secant varieties can be found in \citep{CC02},  \citep{BB} and \citep{LO},  while the papers \citep{LT}, \citep{BCMT}, \citep{BGI}, \citep{CCG}  and \citep{OO} concentrate on effective methods to compute the rank and to compute an explicit decomposition of a form.   In a different direction, the rank of cubic forms associated to canonical curves has been computed in \citep{DZ1}, \citep{DZ2} and \citep{BCN}. 
 
 \section{Apolar Gorenstein subschemes}\label{AGs}

We consider homogeneous polynomials $F\in S:= {\mathbb C}[x_0,\ldots ,x_n]$, and 
consider the dual ring $T:= {\mathbb C}[y_0,\ldots ,y_n]$ acting on $S$ by 
{contraction}:  
\[
y_j(x_i)=\frac{d}{dx_j}(x_i)=\delta_{ij}.
\]
{Differentiation may be used instead of contraction, if care is made with coefficients.}
Let $S_1$ and $T_1$ be the degree 1 parts of $S$ and $T$ respectively. With respect to the action above, $S_{1}$ and $T_{1}$ are 
natural dual spaces and $\langle x_{0},\ldots,x_{n}\rangle$ and 
$\langle y_{0},\ldots,y_{n}\rangle$ are dual bases.   In particular $T$ is 
naturally the coordinate ring of ${\mathbb P}(S_{1})$, the projective space of 
$1$-dimensional subspaces of $S_{1}$, and vice versa.
The annihilator of $F$ is an ideal in  $T$ which we denote by $F^{\bot}\subset 
T$. The quotient $T_{F}:=T/F^{\bot}$ is graded Artinian and Gorenstein since $F^{\bot}$ is homogeneous and $T_F$  is finitely generated as a $\mathbb{C}$-module (\emph{Artinian}) and has a $1$-dimensional socle, the annihilator of the unique maximal ideal (\emph{Gorenstein}).   The socle in $T_F$ is the degree $d$ part of the ring (see e.g. \citep[Lemma 2.14]{IK}).

\begin{defn}
A subscheme $X\subset {\mathbb P}(S_{1})$ is apolar to $F\in S$ if its homogeneous
ideal $I_{X}\subset T$ is contained in $F^{\bot}$. 
\end{defn}

Any apolar subscheme to $F$ of minimal length is locally Gorenstein (\citep[proof of Lemma 2.4]{BB}), therefore we concentrate on finite local Gorenstein schemes.  More precisely, we consider finite subschemes $\Gamma\subset  {\mathbb P}(S_{1})$ isomorphic to ${\rm Spec}R$, where $R$ is a local Artinian Gorenstein ${\mathbb C}$-algebra.  The ring $R$ does not have to be graded.  On the other hand, if $R={\mathbb C}[y_1,\ldots ,y_n]/I$ is a local Artinian Gorenstein algebra, then $I$ is the annihilator $ f^{\bot}$, via {contraction}, of some polynomial $f\in {\mathbb C}[x_1,\ldots ,x_n]$ (cf. \citep[Lemma 1.2]{I}).
If the polynomial $f$ is homogeneous, then $R$ is graded.  This is the case of the form $F$ above.  Now, we consider the affine scheme ${\rm Spec}R$ for possibly inhomogeneous polynomials.
  
In fact, the homogeneous polynomial $F\in S$ admits some natural finite local apolar Gorenstein subschemes.
Let $S_{x_0}:={\mathbb C}[x_1,\ldots ,x_n]$ and  $T_{y_{0}}:={\mathbb C}[y_1,\ldots ,y_n]$.  Consider
the polynomial
$f=F(1,x_{1},\ldots,x_{n})\in S_{x_0}$ and the Artinian 
Gorenstein quotient $T_{f}:=T_{y_{0}}/f^{\bot}$.  We show that the image of the natural embedding ${\rm Spec}(T_{f})\subset \mathbb{P}(S_1)$ is apolar 
$F\in S$. \\
What we have just described for the special case of $f\in S_{x_0}$, the dehomogenisation of $F\in S$ by $x_0$, can be repeated with any other linear form $l\in S_1$.  In fact, $F$ admits a natural apolar Gorenstein subscheme for any linear form in $S$.

Any nonzero linear form $l\in S$ belongs to a basis 
$(l,l_{1},\ldots,l_{n})$ of $S_{1}$, with dual basis 
$(l',l_{1}',\ldots,l_{n}')$ of $T_{1}$.  In particular the 
homogeneous ideal in $T$ of 
the point $[l]\in {\mathbb P}(S_{1})$ is generated by
$\{l_{1}',\ldots,l_{n}'\}$, while $\{l_{1},\ldots,l_{n}\}$ generates 
the ideal of the point $\phi([l])\in{\mathbb P}(T_{1})$, where 
$\phi:{\mathbb P}(T_{1})\to {\mathbb P}(S_{1}),\; y_{i}\mapsto x_{i}, i=0,\ldots,n$. 

The form $F\in S$ defines a hypersurface $\{F=0\}\subset {\mathbb P}(T_{1})$.  
The Taylor expansion of $F$ with respect to the point 
$\phi([l])$ may naturally be expressed in the coordinates functions  $(l,l_{1},\ldots,l_{n})$. 
 Thus there exist $a_0, \ldots , a_d \in \mathbb{C}$ such that
\[
F=a_{0}l^{d}+a_{1}l^{d-1}f_{1}( 
l_{1},\ldots,l_{n})+\cdots+a_{d}f_{d}( 
l_{1},\ldots,l_{n}).
\]
 We denote the corresponding dehomogenisation 
of $F\in S$ with respect to $l\in S_1$ by $F_{l}\in S_{l}$, i.e.   
\[
F_{l}=a_{0}+a_{1}f_{1}( 
l_{1},\ldots,l_{n})+\cdots+a_{d}f_{d}( 
l_{1},\ldots,l_{n}).
\]
Notice that the subscript number of $f_i$ refers to the degree $i$ of the form, distinct from the subscript form of $F_l$ that indicate dehomogenisation with respect to $l$. \\

Also, we denote the subring of $T$ generated by 
$\{l_{1}',\ldots,l_{n}'\}$ by $T_{l'}$.  It is the natural 
coordinate ring of the affine subspace $\{l'\not=0\}\subset {\mathbb P}(S_{1})$.

\begin{lem}\label{dehom} The Artinian Gorenstein scheme $\Gamma(F_{l})$ defined 
    by $F_{l}^{\bot}\subset T_{l'}$ is apolar to $F$, i.e. the 
    homogenisation $(F_{l}^{\bot})^{h}\subset F^{\bot}\subset T$.
    \end{lem}
    \begin{pf} If $g\in F_{l}^{\bot}\subset {\mathbb C}[l_{1}',\ldots,l_{n}']$, 
	then $g=g_{1}+\cdots+g_{r}$ 
	where $g_{i}$ is homogeneous in degree $i$.  Similarly 
	$F_{l}=f=f_{0}+\cdots+f_{d}$.  The annihilation $g(f)=0$ means that for each 
	$e\geq 0$, $\sum_{j}g_{j}f_{e+j}=0$.
	Homogenizing we get
	\[
	g^{h}=G=(l')^{r-1}g_{1}+\cdots+g_{r}, \quad 
	f^{h}=F=l^{d}f_{0}+\cdots+f_{d}
	\]
	and 
	\[
	G(F)=\sum_{e}\sum_{j}l^{d-r-e}g_{j}f_{e+j}=\sum_{e}l^{d-r-e}\sum_{j}g_{j}f_{e+j}=0.
	\]
	\end{pf}

	\begin{rmk}\label{lhom} (Suggested by Mats Boij) The ideal $(F_{l}^{\bot})^{h}$ may 
	     be obtained without dehomogenising $F$.  Write 
	     $F=l^{e}F_{d-e}$, such that $l$ does not divide $F_{d-e}$.
	     Consider the form $
	     l^{d-e}F_{d-e}$ of degree $2(d-e)$.
	    Unless $d-e=0$, i.e. $F=l^{d}$, the degree $d-e$ part of the annihilator 
	    $(l^{d-e}F_{d-e})^{\bot}_{d-e}$ generates an ideal in 
	    $(l)^{\bot}$ and the saturation of $(l^{d-e}F_{d-e})^{\bot}_{d-e}$ coincides with $(F_{l}^{\bot})^{h}$.
	    In fact if $G\in T_{d-e}$ then 
	    \[
	    G(l^{d-e}F_{d-e})=G(l^{d-e})F_{d-e}+lG(l^{d-e-1})F_{d-e}
	    \]
	    so $G(l^{d-e}F_{d-e})=0$ only if $G(l^{d-e})=0$.
	    \end{rmk}

Apolarity was used classically to characterizes powersum 
decompositions of $F$, cf. \citep{IK}, \citep{RS} and the references therein.
In fact, the annihilator of a power of a linear form $l^{d}\in S$ is the ideal of the 
corresponding point $p_{l}\in \mathbb{P}_{T}$ in degrees at most $d$. 
Therefore $F=\sum_{i=1}^{r}l_{i}^{d}$  only if $I_{\Gamma}\subset 
F^{\bot}$ where $\Gamma=\{p_{l_{1}},\ldots,p_{l_{r}}\}\subset 
\mathbb{P}_{T}$.  On the other hand, if $I_{\Gamma,d}\subset 
F^{\bot}_{d}\subset T_{d}$, then any differential form that annihilates each 
$l_{i}^{d}$ also annihilates $F$, so, by duality, $[F]$ must lie in the linear 
span of the $[l_{i}^{d}]$ in ${\mathbb P}(S_{d})$.  Thus $F=\sum_{i=1}^{r}l_{i}^{d}$ if and  only if $I_{\Gamma}\subset 
F^{\bot}$.

The various notions of rank for $F$ listed in the introduction are therefore naturally defined by 
apolarity :
The \emph{cactus rank} $cr(F)$ is defined as
\[
cr(F):={\rm min}\{{\rm length\; of \;a \;scheme \; }\Gamma \;|\; \Gamma\subset {\mathbb P}(T_{1}), 
{\rm dim}\Gamma=0, I_{\Gamma}\subset F^{\bot}\},
\] 
the \emph{smoothable rank} $sr(F)$ is defined as
\[
sr(F):={\rm min}\{{\rm length\; of \;a \;scheme \; }\Gamma \;|\; \Gamma\subset {\mathbb P}(T_{1}) 
\;\;{\rm smoothable}, 
{\rm dim}\Gamma=0, I_{\Gamma}\subset F^{\bot}\}
\]
and the \emph{rank} $r(F)$ is defined as
\[
r(F):={\rm min}\{{\rm length\; of \;a \;scheme \; }\Gamma \;|\; \Gamma\subset {\mathbb P}(T_{1}) \;\;{\rm smooth}, 
{\rm dim}\Gamma=0, I_{\Gamma}\subset F^{\bot}\}.
\]

A smoothable scheme of length $r$ in ${\mathbb P}(T_{1})$  is an element in the irreducible component of the Hilbert scheme containing the smooth schemes of ${\mathbb P}(T_{1})$ of length $r$.

The separate notion of \emph{border rank}, $br(F)$, often considered, is not 
defined by apolarity. 
It is the minimal $r$, such that $F$ is the limit of polynomials of 
rank $r$. These notions of rank coincide 
 with the notions of length of annihilating schemes in 
Iarrobino and Kanev's book \citep[Definition 5.66]{IK}.  Thus cactus 
rank coincides with the scheme length, 
$cr(F)=l{\rm sch}(F)$, and smoothable rank coincides with the smoothable scheme length, $sr(F)=l{\rm schsm}(F)$, while border rank 
coincides with length $br(F)=l(F)$.  In addition they consider the
differential length $l{\rm diff}(F)$, the maximum of the dimensions of the 
space of $k$-th order partials of $F$ as $k$ varies between $0$ and 
deg$F$.  This length is the maximal rank of a catalecticant or Hankel 
matrix at $F$.

Inequalities between these ranks valid for any form $F$ are summarized in  \citep[Lemma 5.17]{IK}.
Clearly, by the definitions above, 
\[
cr(F)\leq sr(F)\leq r(F).
\]
  Furthermore,  
  \[
  br(F)\leq sr(F), \quad {\rm  while}\quad l{\rm diff}(F)\leq br(F)\quad {\rm and}\quad  l{\rm diff}(F)\leq cr(F).
  \]   

For a general form $F$ in $S$ of degree 
$d$ the rank, the smoothable 
rank and the border rank coincide and equals, 
by the Alexander-Hischowitz theorem (see \citep{AH}), 
\[
br(F)=sr(F)=r(F)=\left\lceil\frac{1}{n+1}{n+d \choose d}\right\rceil, 
\]
when $d>2, \;(n,d)\not=(2,4),(3,4),(4,3),(4,4)$.
The local Gorenstein subschemes considered above show that the cactus rank for a general polynomial may be smaller.
Let 

\begin{equation} \label{Nd}
N_{d}=\left\{ \begin{array}{cr}
 2{n+k\choose k}&{\rm when} \; d=2k+1 \\
{n+k\choose k}+{n+k+1 \choose k+1}& {\rm when} \; d=2k+2 \\
\end{array} \right.
\end{equation}
and denote by ${\rm Diff}(F)$ the subspace of $S$ generated by the partials 
of $F$ of all orders, i.e. of order $0,\ldots,d={\rm deg}F$.

\begin{thm}\label{ND} Let $F\in S= {\mathbb C}[x_0,\ldots ,x_n]$ be a homogeneous 
    form of degree $d$, and let $l\in S_{1}= \langle x_{0},\ldots,x_{n}\rangle$ be any 
    linear form. Let $F_{l}$ be a dehomogenisation of $F$ with 
    respect to $l$. Then 
    \[
    cr(F)\leq {\rm dim}_{K}{\rm Diff}(F_{l}).
    \]
    In particular,
    \[
    cr(F)\leq N_{d}.
    \]
    \end{thm}
\begin{pf} According to Lemma \ref{dehom} the subscheme 
    $\Gamma({F_{l}})\subset {\mathbb P}(T_{1})$ is apolar to  $F$.
    The subscheme $\Gamma({F_{l}})$ is affine and has length equal to 
    \[
    {\rm 
    dim}_{k} T_{l'}/F_{l}^{\bot}={\rm dim}_{K}{\rm Diff}(F_{l}).
    \]
    If all the partial derivatives of $F_{l}$ of order at most 
    $\lfloor\frac{d}{2}\rfloor$ are 
    linearly independent, and the partial derivatives of higher order 
    span the space of polynomials of degree at most 
    $\lfloor\frac{d}{2}\rfloor$, 
    then 
    \[
    {\rm dim}_{K}{\rm Diff(F_{l})}= 
    1+n+{n+1 \choose n-1}+\cdots+{n+\lfloor\frac{d}{2}\rfloor \choose n-1}+\cdots+n+1=N_{d}.
    \]
      Clearly this is an 
    upper bound so the theorem follows.
     \end{pf}
   
Local 
apolar subschemes of minimal lengthto some $F$ may not be of the kind $\Gamma 
(F_{l})$, described above.  In fact, even quadratic forms have local apolar 
of length equal to its rank that are not of the kind $\Gamma 
(F_{l})$ (cf. \citep[Corollary 2.7]{RS2}).
   
\begin{question}  What is the cactus rank $cr(n,d)$ for a general 
    form $F\in {\mathbb C}[x_0,\ldots ,x_n]_{d}$?
\end{question}

\section{Cubic forms}\label{cubics}

If $F\in S$ is a general cubic form, then the cactus rank according to Theorem \ref{ND} is at 
most $2n+2$. 

If $F$ is a general reducible cubic form in $S$ and $l$ is a linear factor, then 
$f=F_{l}$ is a quadratic polynomial and  $\Gamma(f)$ is smoothable of length at 
most $n+2$: The partials of a nonsingular quadratic polynomial in $n$ variables 
 form a vector space of dimension $n+2$, so this is the length of  $\Gamma(f)$.
 On the other hand let $E$ be an elliptic normal curve of degree $n+2$ in 
$\mathbb{P}^{n+1}$.  Let $T(E)$ be the homogeneous coordinate ring of 
$E$.  A quotient of $T(E)$ by two general linear forms is Artinian 
Gorenstein with Hilbert function $(1,n,1)$ isomorphic to $T_{q}$ 
for a quadric $q$ of rank $n$.  A quotient of $T(E)$ by two general inhomogeneous 
linear polynomials is the coordinate ring of $n+2$ distinct points. Thus 
$T_{f}$ is isomorphic to $T_{q}$ and $\Gamma(f)$ is smoothable.

\begin{thm}\label{CactusCubics}  For a general cubic form $F\in 
    {\mathbb C}[x_{0},\ldots,x_{n}]$, the cactus rank is
    \[
    cr(F)\leq 2n+2.
    \]
  For a general reducible cubic form $F\in {\mathbb C}[x_{0},\ldots,x_{n}]$ 
    with $n>1$, 
    the cactus rank and the smoothable rank are
    \[
    cr(F)=sr(F)=n+2.
    \]
    \end{thm}
    \begin{pf} It remains to show that for a general reducible 
	cubic form $cr(F)\geq n+2$.  
	On the one hand, if $\Gamma\subset 
	\mathbb{P}_{T}$ has length less than $n+1$ it is contained in 
	a hyperplane, so $I_{\Gamma}\subset F^{\bot}$ only if the 
	latter contains a linear form.  If $\{F=0\}$ is not a cone, 
	this is not the case.   On the other hand, if $\Gamma\subset 
	\mathbb{P}_{T}$ has length $n+1$, then, for the same reason, 
	this subscheme must span $\mathbb{P}_{T}$.  Its ideal in that 
	case is generated by ${n+1 \choose 2}$ quadratic forms.  If 
	$F$ is general, $F^{\bot}_{2}$ is also generated by 
	${n+1 \choose 2}$ quadrics, so they would have to coincide.  
For $F^{\bot}_{2}$ to generate the ideal of a scheme of length $n+1$ is a closed condition on cubic 
	forms $F$.  If $F=x_{0}(x_{0}^{2}+\cdots+x_{n}^{2})$, then 
	\[
	F^{\bot}_{2}=\langle y_{1}y_{2},\ldots,y_{n-1}y_{n}, 
	y_{0}^{2}-y_{1}^{2},\ldots,y_{0}^{2}-y_{n}^{2}\rangle.
	\]
	In particular ${\rm dim}F^{\bot}_{2}={n+1 \choose 2}$,
	but the quadrics $F^{\bot}_{2}$ do not have any common zeros, so $cr(F)\geq 
	n+2$.  The  general reducible cubic must therefore also have 
	cactus rank at least $n+2$ and the theorem follows.
	\end{pf}
	\begin{rmk} By \citep[Theorem 1.3]{LT} the lower bound for 
	    the rank of a reducible cubic form that depends on $n+1$ 
	    variables and not less, is $2n$.  
	    \end{rmk}
	   If $F=x_{0}F_{1}(x_{1},\ldots,x_{n})$ where $F_{1}$ is a 
	   quadratic form of rank $n$, then 
	   \[
	   cr(F)=sr(F)=n+1,
	   \]
	    the same as for a Fermat cubic, while the rank is at least $2n$.
	 
	 We give an example with $cr(F)= l{\rm diff}(F)=n+1<sr(F)$.

	 \begin{exa}   
	   	 Let $G\in {\mathbb C}[x_{1},\ldots,x_{m}]$ be a cubic form such that the 
	 scheme $\Gamma({G})={\rm 
	 Spec}({\mathbb C}[y_{1},\ldots,y_{m}]/G^{\bot})$ has length $2m+2$ and is not 
	 smoothable.  By \citep[Section 4A]{I84} examples occur for $m\geq 6$.  Denote by $G_{1}=y_{1}(G),\ldots,G_{m}=y_{m}(G)$ the first order 
	 partials of $G$.  Let 
	 \[
	 F=G+x_{0}x_{1}x_{m+1}+x_0x_2x_{m+2}+\cdots 
	 +x_{0}x_{m}x_{2m}+x_{0}^{2}x_{2m+1}\in 
	 {\mathbb C}[x_{0},\ldots,x_{2m+1}].
	 \]
	   Then 
	 \[
	 F_{x_{0}}=G+x_{1}x_{m+1}+\cdots 
	 +x_{m}x_{2m}+x_{2m+1}
	 \]
	  and
	 \[
	 {\rm Diff}(F_{x_{0}})=\langle F_{x_{0}}, G_{1}+x_{m+1},\ldots ,
	 G_{m}+x_{2m},x_{1},\ldots,x_{m}, 1\rangle
	 \]
	 so ${\rm dim}{\rm Diff}(F_{x_{0}})=2m+2$.  Therefore 
	 $\Gamma(F_{x_{0}})$ is apolar to $F$ and computes the 
	 cactus rank of $F$.  Since $\{F=0\}$ is not a cone, 
	 $\Gamma(F_{x_{0}})$ is nondegenerate, so its homogeneous ideal is 
	 generated by the quadrics in the ideal of $F^{\bot}$.  
	 In particular  $\Gamma(F_{x_{0}})$ is the unique apolar 
	 subscheme of length $2m+2$.  Since this is not smoothable, 
	 the smoothable rank is strictly bigger.
	 \end{exa}

By Theorem \ref{CactusCubics} the cactus rank of a generic cubic form $F\in {\mathbb C}[x_0, \ldots , x_n]$
is at most $2n+2$. The first $n$ for which $2n+2$ is smaller than  the rank  
$r(F)=\lceil\frac{1}{n+1}{n+3 \choose 3}\rceil$  of the generic cubic form in $n+1$ variables 
is $n=8$, where  $r(F)=19$ and $ cr(F)\leq 18$.

\begin{conjecture}
The cactus rank $cr(F)$ of a general homogeneous cubic $F\in k [x_0, \ldots , x_n]$  equals the rank when $n\leq 7$ and equals $2n+2$ when $n\geq 8$. \end{conjecture}

For a general cubic  form, the rank is $\leq 10$ when $n\leq 5$, while it is $12$ when $n=6$.
Now, any local Artinian Gorenstein scheme of length 
   at most $10$ is smoothable (cf. \citep{CN10}), so the conjecture holds for $n\leq 5$.  Casnati and Notari has recently extended their result to length at most $11$,  (cf. \citep{CN11}), which means that the conjecture holds also when $n=6$.  There are nonsmoothable local Gorenstein algebras of length $14$ (cf. \citep{I84}), so for  $n\geq 7$  a different argument is needed to confirm or disprove the conjecture.

\begin{ack}
The authors would like to thank 
the Institut Mittag-Leffler (Djursholm, Sweden)
for their support and hospitality, and Tony Iarrobino for helpful comments on Gorenstein algebras. {We also thank Joachim Jelisiejew for pointing out that Lemma \ref{dehom} is correct as stated only when using apolarity by contraction (as reflected in the current version of this paper, updated in May 2024) instead of  apolarity by differentiation (as it was in the originally published version).}
\end{ack}

\bibliographystyle{elsart-harv}

\end{document}